\documentclass{amsart}
\usepackage{amssymb}
\usepackage{amscd}
\usepackage{hyperref}

\newcommand{\isomto}{\overset{\thicksim}{\longrightarrow}}

\newcommand{\surjto}{\relbar\joinrel\twoheadrightarrow}

\newcommand{\longto}{\longrightarrow}
\newcommand{\shortinjto}{\lhook\joinrel\rightarrow}
\def\bQ{{\mathbf{Q}}}  \def\bZ{{\mathbf{Z}}}
 \def\bP{{\mathbf{P}}} \def\bF{{\mathbf{F}}}
\def\bG{{\mathbf{G}}}  

\def\cO{{\mathcal{O}}}  \def\cV{{\mathcal{V}}}
 
\def\cHom{{{\mathcal Hom}}}

\def\fp{{\mathfrak{p}}} \def\fm{{\mathfrak{m}}} \def\fq{{\mathfrak{q}}}
\def\fP{{\mathfrak{P}}}

\DeclareMathOperator\Pic{{Pic}}

\DeclareMathOperator\Spec{{Spec}}

\DeclareMathOperator\End{{End}}

\DeclareMathOperator\Hom{{Hom}}
\DeclareMathOperator\Ext{{Ext}}
\DeclareMathOperator\Gal{{Gal}}

\DeclareMathOperator\dlog{{dlog}}

\usepackage[OT2,T1]{fontenc}
\DeclareSymbolFont{cyrletters}{OT2}{wncyr}{m}{n}
\DeclareMathSymbol{\Sha}{\mathalpha}{cyrletters}{"58}

\DeclareMathOperator\car{{char}}

\DeclareMathOperator\rk{{rk}}

\def\sep{{\rm sep}}
\def\diff{{\rm d}}
\def\H{{\rm H}}

\def\R{{\rm R}}
\def\fl{{\rm fl}}
\def\et{{\rm et}}

\def\Frob{{\rm Frob}}
\def\cart{{\rm c}}
\def\rmD{{\rm D}}
\def\rmF{{\rm F}}

\def\phi{\varphi}

\theoremstyle{plain}
\newtheorem{theorem}{Theorem}

\newtheorem{lemma}{Lemma}
\newtheorem{proposition}{Proposition}

\newtheorem{question}{Question}
\theoremstyle{definition}
\newtheorem{definition}{Definition}
\newtheorem{example}[subsection]{Example}
\newtheorem{remark}{Remark}

\title{A Herbrand-Ribet theorem for function fields}
\author{Lenny Taelman}
\address{Mathematisch Instituut, P.O. Box 9512, 2300 RA Leiden, The Netherlands}
\email{lenny@math.leidenuniv.nl}

\begin{document}

\begin{abstract}
We prove a function field analogue of the Herbrand-Ribet theorem on cyclotomic number fields. The Herbrand-Ribet theorem can be interpreted as a result about cohomology with
$\mu_p$-coefficients over the splitting field of $\mu_p$, and in our analogue both occurrences of $\mu_p$ are replaced with the $\fp$-torsion scheme of the Carlitz module for a prime $\fp$ in $\bF_q[t]$.
\end{abstract}

\maketitle

\tableofcontents

\section{Introduction and statement of the theorem}

Let $p$ be a prime number, $F=\bQ(\zeta_p)$ the $p$-th cyclotomic number field and
$\Pic \cO_F$ its class group. Then $\bF_p \otimes_\bZ \Pic \cO_F$
decomposes in eigenspaces under the action of the Galois group $\Gal(F/\bQ)$ as
\[
	\bF_p \otimes_\bZ \Pic \cO_F =
	\bigoplus_{n=1}^{p-1} \left(\bF_p \otimes_\bZ \Pic \cO_F\right)(\omega^n)
\]
where $\omega\colon \Gal(F/\bQ) \to \bF_p^\times$ is the cyclotomic character. 

If $n$ is a nonnegative integer we denote by $B_n$ the $n$-th Bernoulli number, defined by the identity
\[
	\frac{z}{\exp z-1} = \sum_{n=0}^\infty B_n \frac{z^n}{n!}.
\]
If $n$ is smaller than $p$ then $B_n$ is $p$-integral. The \emph{Herbrand-Ribet theorem} \cite{Herbrand32} \cite{Ribet76} states that if $n$ is even and $1<n<p$ then
\[
	\left(\bF_p \otimes_\bZ \Pic \cO_F\right)(\omega^{1-n}) \neq 0
	\,\text{ if and only if } \, p \mid B_n.
\]
The \emph{Kummer-Vandiver conjecture} asserts that for all odd $n$ we have
\[
	\left(\bF_p \otimes_\bZ \Pic \cO_F\right)(\omega^{1-n})=0.
\]

\bigskip

In this paper we will state and prove a function field analogue of the Herbrand-Ribet theorem and state an analogue of the Kummer-Vandiver conjecture.

Let $k$ be a finite field of $q$ elements and $A=k[t]$ the polynomial ring in one variable $t$ over $k$. Let $K$ be the fraction field of $A$.

\begin{definition}\label{defcarlitz}
The \emph{Carlitz module} is the $A$-module scheme $C$ over $\Spec A$ whose underlying $k$-vectorspace scheme is the additive group $\bG_a$ and whose $k[t]$-module structure is given by
the $k$-algebra homomorphism
\[
	\phi\colon A \to \End(\bG_a), \, t \mapsto t + F,
\]
where $F$ is the $q$-th power Frobenius endomorphism of $\bG_a$. 
\end{definition}

The Carlitz module is in many ways an $A$-module analogue of the $\bZ$-module scheme $\bG_m$.
For example, the $\Gal(K^\sep/K)$-action on torsion points is formally similar to the
$\Gal(\bar{\bQ}/\bQ)$-action on roots of unity:

\begin{proposition}[{\cite[\S 7.5]{Goss96}}]
Let $\fp\subset A$ be a nonzero prime ideal, then $C[\fp](K^\sep) \cong A/\fp$ and the resulting
Galois representation
\[
	\rho\colon \Gal(K^\sep/K) \longto (A/\fp)^\times.
\]
satisfies
\begin{enumerate}
\item if a prime $\fq\subset A$ is coprime with $\fp$ then
$\rho$ is unramified at $\fq$ and maps a Frobenius element to the class in
$(A/\fp)^\times$ of the monic generator of $\fq$;
\item $\rho( D_\infty ) = \rho( I_\infty ) = k^\times$;
\item $\rho( D_\fp ) = \rho( I_\fp ) = (A/\fp)^\times $,
\end{enumerate}
where the $D$'s and $I$'s denote decomposition and inertia subgroups.\qed
\end{proposition}

Now fix a nonzero prime ideal $\fp \subset A$ of degree $d$. 
Let $L$ be the splitting field of $\rho$. Then $L/K$ is unramified outside $\fp$ and $\infty$, and 
$\rho$ induces an isomorphism $\chi\colon G=\Gal(L/K)\isomto (A/\fp)^\times$.

Let $R$ be the normalization of $A$ in $L$ and $Y=\Spec R$. Let $Y_\fl$ be the flat site on $Y$: the category of schemes locally of finite type over $Y$, with covering families being the jointly surjective families of flat morphisms.

The $\fp$-torsion $C[\fp]$ of $C$ is a finite flat group scheme of rank $q^d$ over $\Spec A$. Let $C[\fp]^\rmD$ be the Cartier dual of $C[\fp]$ and consider the decomposition
\[
	\H^1( Y_\fl, C[\fp]^\rmD ) = \bigoplus_{n=1}^{q^d-1} \H^1( Y_\fl, C[\fp]^\rmD )(\chi^n)
\]
of the $A/\fp$-vector space $\H^1( Y_\fl, C[\fp]^\rmD )$ under the natural action of $G$.

Our analogue of the Herbrand-Ribet theorem will give a criterion for the vanishing of some of these eigenspaces in terms of divisibility by $\fp$ of the so-called Bernoulli-Carlitz numbers,
which we now define.  

The \emph{Carlitz exponential} is the unique power series $e(z) \in K[[z]]$
which satisfies 
\begin{enumerate}
\item $e(z) = z + e_1 z^q + e_2 z^{q^2} + \cdots$ with $e_i\in K$;
\item $e(tz) = e(z)^q + te(z)$.
\end{enumerate}
The Carlitz exponential converges on any finite extension of $K_\infty$ and on an algebraic
closure $\bar K_\infty$ it defines a surjective homomorphism of
$A$-modules
\[
	e \colon \bar K_\infty \surjto C(\bar K_\infty)
\]
whose kernel is discrete and free of rank $1$. We define $BC_n \in K$ by the power series identity
\[
	\frac{z}{e(z)} = \sum_{n=0}^\infty BC_n z^n.
\]
If $n$ is not divisible by $q-1$ then $BC_n$ is zero. If $n$ is less than $q^d$ then $BC_n$ is $\fp$-integral.

\begin{theorem} \label{mainthm}
Let $0 < n < q^d-1$ be divisible by $q-1$.  
Then $\fp$ divides $BC_n$ if and only if $\H^1( Y_\fl, C[\fp]^\rmD)( \chi^{n-1} )$ is nonzero.
\end{theorem}

This is the analogue of the Herbrand-Ribet theorem. The proof is given in section
\ref{secsketch}, modulo auxiliary results which are proven in sections \ref{secconstantcoeff}--\ref{secvanishing}.

In this context a natural analogue of 
the Kummer-Vandiver conjecture is the following:

\begin{question}\label{qvandiver}
Does $\H^1( Y_\fl, C[\fp]^\rmD)( \chi^{n-1} )$ vanish if $n$ is not divisible by $q-1$?
\end{question}

By computer calculation we have verified that these groups indeed vanish for small $q$ and primes $\fp$ of small degree, \emph{see} \S \ref{sectables}.
However, if one believes in a function field version of Washington's heuristics \cite[\S 9.3]{Washington97} then one should expect that counterexamples do exist, but are very sparse, 
making it difficult to obtain convincing numerical evidence towards Question \ref{qvandiver}.

\begin{remark}
Our $BC_n$ differ from the commonly used \emph{Bernoulli-Carlitz} numbers by
a \emph{Carlitz factorial} factor (see for example \cite[\S 9.2]{Goss96}). This factor is innocent for our purposes since it is a unit at $\fp$ for $n<q^d$.
\end{remark}
\begin{remark}
Let $p$ be an odd prime number, $F=\bQ(\zeta_p)$ and $D = \Spec \cO_F$.
Global duality \cite{Mazur73} provides a perfect pairing between
\[
	\bF_p \otimes_\bZ \Pic D = \Ext^2_{D_\et}( \bZ/p\bZ, \bG_{m,D} )
\] 
and
\[
	\H^1( D_\et, \bZ/p\bZ ) = \H^1( D_\fl, \bZ/p\bZ ).
\]
The Herbrand-Ribet theorem thus says that (for $1<n<p-1$ even)
\[
	p \mid B_n \, \text{ if and only if } \,
	\H^1( D_\fl, \mu_p^\rmD )(\chi^{n-1}) \neq 0,
\]
in perfect analogy with the statement of Theorem \ref{mainthm}.
\end{remark}

\begin{remark}
The analogy goes even further.
In \cite{Taelman10b} and \cite{Taelman11} we have defined a finite $A$-module 
$H(C/R)$, analogue of the class group $\Pic \cO_F$, and although we will not use this in the proof
of Theorem \ref{mainthm}, we show in Section \ref{complement} of this paper
 that there are 
canonical isomorphisms
\[
	A/\fp \otimes_A H(C/R) \isomto \Hom( \H^1( Y_\fl, C[\fp]^\rmD ), \bF_p ).
\]
\end{remark}

\begin{remark}
A more naive attempt to obtain a function field analogue of the Herbrand-Ribet theorem would be to compare the $\fp$-divisibility of the Bernoulli-Carlitz numbers with the $p$-torsion of 
the divisor class groups of $Y$ and $L$ (where $p$ is the characteristic of $k$). In other words, to consider cohomology with $\mu_p$-coefficients on the curves defined by the splitting of $C[\fp]$. Several results of this kind have in fact been obtained by  Goss \cite{Goss87}, Gekeler \cite{Gekeler90}, Okada \cite{Okada91}, and Angl\`{e}s \cite{Angles01}, but there appears to be no complete analogue of the Herbrand-Ribet theoreom in this context.

In the proof of Theorem \ref{mainthm} we will see that the $A$-module 
$\H^1( Y_\fl, C[\fp]^\rmD )$ and the group $(\Pic Y)[p]$ are related, and this relationship might shed some new light on these older results. 
\end{remark}

\begin{remark}
I do not know if there is a relation between Question \ref{qvandiver} and Anderson's 
analogue of the Kummer-Vandiver conjecture \cite{Anderson96}.
\end{remark}

\emph{Acknowledgements.}  I am grateful to David Goss for his insistence that I consider the decomposition in isotypical components of the ``class module'' of \cite{Taelman10b} and \cite{Taelman11}, and to the referee for several useful suggestions.  The author is supported by a grant of the Netherlands Organisation for Scientific Research (NWO).

\section{Tables of small irregular primes}\label{sectables}

The results of section \ref{complement} indicate a method for computing 
the modules $\H^1( Y_\fl, C[\fp]^\rmD )$ with their $G$-action in terms of finite-dimensional vector spaces of differential forms on the compactification $X$ of $Y$.

Assisted by the computer algebra package MAGMA we were able to compute them in the following ranges:
\begin{enumerate}
\item $q=2$ and $\deg\fp \leq 5$;
\item $q=3$ and $\deg\fp \leq 4$;
\item $q=4$ and $\deg\fp \leq 3$;
\item $q=5$ and $\deg\fp \leq 3$.
\end{enumerate}

In all these cases $\H^1( Y_\fl, C[\fp]^\rmD )$ turns out to be at most one-dimensional,
and to fall in the $\chi^{n-1}$-component with $n$ divisible by $q-1$ (and hence with $\fp$ dividing $BC_n$.) In particular we have not found any counterexamples to Question \ref{qvandiver}.

In tables 1--3 we list all cases where the cohomology group is nontrivial. 
For $q=5$ and $\deg \fp \leq 3$ the group turns out to vanish. In the middle columns, only $n$ in the range $1 \leq n < q^{\deg \fp}$ are printed.

\begin{table}[h!]\label{tab2}
\begin{tabular}{c|c|c}
 $\fp$ & $\{ n : \fp \mid BC_n \}$ & $\dim \H^1( Y_\fl, C[\fp]^\rmD )$ \\
\hline
$(t^4 + t + 1)$ & $\{ 9 \}$  & $1$ 
\end{tabular}
\smallskip
\caption{All irregular primes in $\bF_2[t]$ of degree at most 5}
\end{table}

\begin{table}[h]\label{tab3}
\begin{tabular}{c|c|c}
 $\fp$ & $\{ n : \fp \mid BC_n \}$ & $\dim \H^1( Y_\fl, C[\fp]^\rmD )$ \\
\hline
$(t^3 - t + 1)$ & $\{ 10 \}$  & $1$ \\ 
$(t^3 - t - 1)$ & $\{ 10 \}$  & $1$ \\ 
$(t^4 - t^3 + t^2 + 1)$ & $\{ 40 \}$ & $1$ \\
$(t^4 - t^2 - 1)$ & $\{ 32 \}$ & $1$ \\
$(t^4 - t^3 - t^2 + t - 1)$ & $\{ 32 \}$ & $1$ \\
$(t^4 + t^3 + t^2 + 1)$ & $\{ 40 \}$ & $1$ \\
$(t^4 + t^3 - t^2 - t - 1)$ & $\{ 32 \}$ & $1$ \\
$(t^4 + t^2 - 1)$ & $\{ 40 \} $ & $1$ 
\end{tabular}
\smallskip
\caption{All irregular primes in $\bF_3[t]$ of degree at most 4 }
\end{table}

\begin{table}[h]\label{tab4}
\begin{tabular}{c|c|c}
 $\fp$ & $\{ n : \fp \mid BC_n \}$ & $\dim \H^1( Y_\fl, C[\fp]^\rmD )$ \\
\hline
$(t^3 + t^2 + t + \alpha)$ & $\{ 33 \}$  & $1$ \\ 
$(t^3 + t^2 + t + \alpha^2)$ & $\{ 33 \}$  & $1$ \\ 
$(t^3 + \alpha)$ & $\{33\}$ & $1$ \\
$(t^3 + \alpha^2)$ & $\{33\}$ & $1$ \\
$(t^3 + \alpha^2 t^2 + \alpha t + \alpha^2)$ & $\{33\}$ & $1$ \\
$(t^3 + \alpha t^2 + \alpha^2 t + \alpha)$ & $\{33\}$ & $1$ \\
$(t^3 + \alpha t^2 + \alpha^2 t + \alpha^2)$ & $\{33\}$ & $1$ \\
$(t^3 + \alpha^2 t^2 + \alpha t + \alpha)$ & $\{33\}$ & $1$ 
\end{tabular}
\smallskip
\caption{All irregular primes in $\bF_4[t]$ of degree at most 3 (with
	$\bF_4=\bF_2(\alpha)$).}
\end{table}

\section{Notation and conventions}

\emph{Basic setup.} $k$ is a finite field of $q$ elements, $p$ its characteristic. $A = k[t]$ and $\fp \subset A$ a nonzero prime. These data are fixed throughout the text. We denote by $d$ the degree of $\fp$, so that $A/\fp$ is a field of $q^d$ elements.

\emph{The Carlitz module.} The Carlitz module is the $A$-module scheme $C$ over $\Spec A$ defined in Definition \ref{defcarlitz}.

\emph{Cyclotomic curves and fields.} $K$ is the fraction field of $A$, and $L/K$ the splitting field of $C[\fp]_K$. The integral closure of $A$ in $L$ is denoted by $R$, and $Y=\Spec R$. We denote by $\fP \subset R$ the unique prime lying above $\fp \subset A$.

\emph{Sites.} For any scheme $S$ we denote by $S_\et$ the \emph{small \'etale site} on $S$ and by $S_\fl$ the \emph{flat site} in the sense of \cite{Milne86b}: the category of schemes locally of finite type over $S$ where covering families are jointly surjective families of flat morphisms. For every $S$ there is a canonical morphism of sites $f\colon S_\fl \to S_\et$. Any commutative group scheme over $S$ defines a sheaf of abelian groups on $S_\fl$ and on $S_\et$.

\emph{Cartier dual.} If $G$ is a finite flat commutative group scheme, then $G^\rmD$ denotes the Cartier dual of $G$. 

\emph{Frobenius and Cartier operators.} For any $k$-scheme $S$ we denote by
\[
	\rmF\colon \bG_{a,S} \to \bG_{a,S},\, x \mapsto x^q
\]
the $q$-power Frobenius endomorphism of sheaves on $S_\fl$ or $S_\et$, and by
\[
	\cart\colon \Omega_S \to \Omega_S
\]
the $q$-Cartier operator of sheaves on $S_\et$. If $q=p^r$ with $p$ prime this is the $r$-th power
of the usual Cartier operator. The endomorphism $\cart$ satisfies
$\cart( f^q\omega ) = f \cart( \omega )$
for all local sections $f$ of $\cO_S$ and $\omega$ of $\Omega_S$. In particular it is $k$-linear.

\section{Overview of the proof}\label{secsketch}

Choose a generator $\lambda$ of $C[\fp](L)$. It defines a map of finite
flat group schemes
\[
	\lambda\colon (A/\fp)_Y \longto C[\fp]_Y
\]
which is an isomorphism over $Y-\fP$. It induces a map of Cartier duals
\[
	C[\fp]^\rmD_Y \longto (A/\fp)_Y^\rmD
\]
and a map on cohomology
\[
	\H^1( Y_\fl, C[\fp]^\rmD ) \longto \H^1( Y_\fl, (A/\fp)^\rmD ).
\]
This map is not $G$-equivariant (since $\lambda$ is not $G$-invariant), but rather restricts for every $n$ to a map
\begin{equation}\label{maplambda}
	\H^1( Y_\fl, C[\fp]^\rmD )(\chi^{n-1}) 
	\overset{\lambda}{\longto} \H^1( Y_\fl, (A/\fp)^\rmD )(\chi^n).
\end{equation}
We will see in section \ref{secconstantcoeff} that there is a natural $G$-equivariant 
isomorphism
\[
	\H^1( Y_\fl, (A/\fp)^\rmD ) \isomto A/\fp \otimes_k \Omega_R^{\cart=1}
\]
where $\Omega_R^{\cart=1}$ is the $k$-vector space of $q$-Cartier invariant K\"ahler differentials.
Also, we will see that the Kummer sequence induces a short exact sequence
\begin{equation}\label{kummer}
	0 \longto   A/\fp \otimes_\bZ \Gamma( Y, \cO_Y^\times )
	\overset{\dlog}{\longto} A/\fp \otimes_k \Omega_R^{\cart=1} \longto
	A/\fp \otimes_{\bF_p} (\Pic Y)[p] \longto 0.
\end{equation}
Note that the residue field of the completion $R_\fP$ is $A/\fp$, so $R_\fP$ is naturally an
$A/\fp$-algebra. In particular, for all $m$ the $R$-module
$\Omega_R/\fP^m\Omega_R$ is naturally an $A/\fp$-module. Using this the quotient map
$\Omega_R \surjto \Omega_R/\fP^m \Omega_R$ extends to an $A/\fp$-linear map
\[
	A/\fp\otimes_k\Omega_R \surjto \Omega_R/\fP^m\Omega_R.
\]

In section \S \ref{seccomparing} we will use the results on  
flat duality of Artin and Milne \cite{Artin76} to show the following.
\begin{theorem}\label{thmdiff}
For all $n$ the sequence of $A/\fp$-vector spaces
\[
	0 \longto \H^1( Y_\fl, C[\fp]^\rmD )(\chi^{n-1}) \overset{\lambda}{\longto}
		A/\fp\otimes_k\Omega_R^{\cart=1}(\chi^n) \longto \Omega_R/\fP^{q^d}\Omega_R
\]
is exact.
\end{theorem}

The function $\lambda$ is invertible on $Y-\fP$. 
Consider the decomposition of
$1\otimes \lambda \in A/\fp \otimes_\bZ \Gamma(Y-\fP, \cO_Y^\times )$ in isotypical components:
\[
	 1\otimes \lambda = \sum_{n=1}^{q^d-1} \lambda_n 
	 \quad \text{ with } \quad
	 \lambda_n \in A/\fp \otimes_\bZ \Gamma(Y-\fP, \cO_Y^\times )(\chi^n).
\]
The homomorphism $\dlog\colon R^\times \to \Omega_R$ extends to an $A/\fp$-linear map
\[
	A/\fp \otimes_\bZ \Gamma( Y, \cO_Y^\times ) \longto \Omega_R.
\]
Inspired by Okada's construction \cite{Okada91} of a Kummer homomorphism for function fields we 
prove in section \ref{seccandidate} the following result.

\begin{theorem}\label{thmcandidate}
If $1 \leq n < q^d-1$ then $\lambda_n \in A/\fp \otimes_\bZ \Gamma( Y, \cO_Y^\times )$ 
and the following are equivalent:
\begin{enumerate}
\item $\fp$ divides $BC_n$;
\item $\dlog \lambda_n$ lies in the kernel of
	$A/\fp\otimes_k\Omega_R \to \Omega_R/\fP^{q^d}\Omega_R$.
\end{enumerate}
\end{theorem}

It may (and does) happen that $\lambda_n$ vanishes for some $n$ divisible by $q-1$. However, the following theorem provides us with sufficient control over the vanishing of $\lambda_n$.

\begin{theorem}\label{thmvanishing}
If $n$ is divisible by $q-1$ but not by $q^d-1$ then the following are equivalent:
\begin{enumerate}
\item $\lambda_n=0$;
\item $A/\fp\otimes_{\bF_p}(\Pic Y)[p](\chi^n)\neq 0$.
\end{enumerate}
\end{theorem}

The proof is an adaptation of work of Galovich and Rosen \cite{Galovich81}, and uses $L$-functions in characteristic $0$. It is given in section \ref{secvanishing}.

\medskip

Assuming the three theorems above, we can now prove the main result.

\begin{proof}[Proof of Theorem \ref{mainthm}]
Assume $q-1$ divides $n$ and $\fp$ divides $BC_n$. We need to show that $\H^1( Y_\fl, C[\fp]^\rmD )(\chi^{1-n})$ is nonzero. Being a (component of) a differential logarithm
$\dlog \lambda_n$ is Cartier-invariant and Theorem \ref{thmcandidate} tells us that 
\[
	\dlog \lambda_n \in A/\fp \otimes_k\Omega_R^{\cart=1}(\chi^n)
\]
maps to $0$ in $\Omega_R/\fP^{q^d}\!\Omega_R$. If $\lambda_n\neq 0$ then by Theorem \ref{thmdiff} we conclude that $\H^1( Y_\fl, C[\fp]^\rmD )(\chi^{n-1})$ is nonzero and we are done.
So assume that $\lambda_n=0$. Consider the short exact sequence (\ref{kummer}).
By Theorem \ref{thmvanishing} we have that
\[
	\dim_{A/\fp} A/\fp \otimes_{\bF_p} (\Pic Y)[p](\chi^n) \geq 1,
\]
and since $ A/\fp \otimes_\bZ \Gamma( Y, \cO_Y^\times )(\chi^n) $ is one-dimensional, 
we find that
\[
	\dim_{A/\fp} A/\fp \otimes_k \Omega_R^{\cart=1} (\chi^n) \geq 2.
\]
But $\Omega_R/\fP^{q^d}\Omega_R (\chi^n) $ is one-dimensional, so it follows from
Theorem \ref{thmdiff} that
\[
	\H^1( Y_\fl, C[\fp]^\rmD)(\chi^{n-1}) \neq 0.
\]

Conversely, assume that $q-1$ divides $n$ and $\fp$ does \emph{not} divide $BC_n$. Then
Theorem \ref{thmcandidate} guarantees that $\dlog\lambda_n$ is nonzero and it follows from Theorem
\ref{thmvanishing} and the short exact sequence (\ref{kummer}) that 
\[
	\dim A/\fp \otimes_k \Omega_{R}^{\cart=1}(\chi^n) = 1.
\]
Therefore $A/\fp \otimes_k \Omega_R^{\cart=1}(\chi^n)$ is generated by $\dlog\lambda_n$ and since
the image of $\dlog \lambda_n$ in $\Omega_R/\fP^{q^d}\Omega_R$ is nonzero we conclude 
from Theorem \ref{thmdiff} that
$\H^1( Y_\fl, C[\fp]^\rmD )( \chi^{n-1} ) $ vanishes.
\end{proof}

\section{Flat duality}

In this section we summarize some of the results of Artin and Milne \cite{Artin76} on duality for flat cohomology in characteristic $p$.

Let $S$ be a scheme over $k$ and $\cV$ a quasi-coherent $\cO_S$-module. Then the pull-back 
$F^\ast \cV$ of $\cV$ under $F\colon S\to S$ is a quasi-coherent $\cO_S$-module and there is a $k$-linear (typically \emph{not} $\cO_S$-linear) isomorphism 
\[
	F\colon \cV \longto F^\ast \cV
\]
of sheaves on $S_\fl$.  

If $S$ is smooth of relative dimension $1$ over $k$ then the $q$-Cartier operator induces a canonical map
\[
	\cart \colon \cHom( F^\ast \cV, \Omega_{S/k} )
		\longto \cHom( \cV, \Omega_{S/k} )
\]
of sheaves on $S_\et$.

Recall that the we denote the canonical map $S_\fl \to S_\et$ by $f$.

\begin{theorem}[Artin \& Milne]\label{thmduality}
Let $S$ be smooth of relative dimension $1$ over $\Spec k$. Let 
\begin{equation}\label{resolution}
	0 \longto G \longto \cV \overset{\alpha-F}{\longto} F^\ast\cV \longto 0
\end{equation}
be a short exact sequence of sheaves on $S_\fl$ with
\begin{enumerate}
\item $\cV$ a locally free coherent $\cO_S$-module;
\item $\alpha\colon \cV \to F^\ast\cV$ a morphism of $\cO_S$-modules.
\end{enumerate}
Then $G$ is a finite flat group scheme and there is a short exact sequence
\begin{equation}\label{dualseq}
	0 \longto \R^1\!f_\ast G^\rmD \longto
		\cHom( F^\ast\cV, \Omega_{S/k} )
		\overset{\alpha-\cart}{\longto}
		\cHom( \cV, \Omega_{S/k} )
		\longto 0
\end{equation}
of sheaves on $S_\et$, functorial in (\ref{resolution}). Moreover, for all $i \neq 1$ one has 
$\R^i\!f_\ast G^\rmD = 0 $.	 
\end{theorem}

\begin{proof}
Locally on $S$, we have that $G$ is given as a closed subgroup scheme of $\bG_a^n$ defined by equations of the form $FX-\alpha X=0$. In particular $G$ is flat of degree $q^{\rk \cV}$.
The Cartier dual $G^\rmD$ of $G$ is a finite flat group scheme of height $1$.

If $q$ is prime then the existence of (\ref{dualseq}) is shown in \cite[\S 2]{Artin76}. One can deduce the general case from this as follows. Assume $n$ is a positive integer, and assume given a short exact sequence 
\[
	0 \longto G \longto \cV \overset{\alpha-F^n}{\longto} (F^n)^\ast\cV \longto 0
\]
of sheaves on $S_\fl$, with $\alpha\colon \cV \to (F^n)^\ast \cV$ a $\cO_{S}$-linear map.
Define
\[
	\cV' := \cV \oplus F^\ast \cV \oplus \cdots \oplus (F^{n-1})^\ast \cV.
\]
The map $\alpha$ induces an $\cO_S$-linear map
\[
	\alpha' \colon \cV' \longto F^\ast \cV'
\]
defined by mapping the component $\cV$ to the component $(F^{n})^\ast \cV$ using $\alpha$, and mapping all other components to zero. We thus have a short exact sequence
\[
	0 \longto G \longto \cV' \overset{\alpha'-F}{\longto} F^\ast \cV' \longto 0
\]
and one deduces the theorem for $F^n$ from the theorem for $F$.
\end{proof}

\begin{example}
If $k=\bF_p$ then the Artin-Schreier exact sequence
\[
	0 \longto \bZ/p\bZ \longto \bG_a \overset{1-F}{\longto} \bG_a \longto 0
\]
on $S_\fl$ induces a dual exact sequence
\[
	0 \longto \R^1\!f_\ast\mu_p \longto
	\Omega_{S/k} \overset{1-\cart}{\longto} \Omega_{S/k} \longto 0
\]
on $S_\et$, and the exact sequence
\[
	0 \longto \alpha_p \longto \bG_a \overset{-F}{\longto} \bG_a \longto 0
\]
on $S_\fl$ induces a dual exact sequence
\[
	0 \longto \R^1\!f_\ast\alpha_p \longto
	\Omega_{S/k} \overset{-\cart}{\longto} \Omega_{S/k} \longto 0
\]
on $S_\et$.
\end{example}

\section{Flat cohomology with $(A/\fp)^\rmD$ coefficients}\label{secconstantcoeff}

The constant sheaf $A/\fp$ on $Y_\fl$ has a resolution
\[
	0 \longto A/\fp \longto A/\fp \otimes_k \bG_{a,Y} 
		\overset{1-1\otimes F}{\longto} A/\fp \otimes_k \bG_{a,Y}
			\longto 0
\]
so by Theorem \ref{thmduality} we have
$\R^i\!f_\ast (A/\fp)^\rmD = 0$ for $i\neq 1$, and $\R^1\!f_\ast (A/\fp)^\rmD$ sits in a
short exact sequence
\[
	1 \longrightarrow \R^1\!f_\ast (A/\fp)^\rmD \longrightarrow
	A/\fp\otimes_k \Omega_Y
	\overset{1\otimes\cart-1}{\longrightarrow}
	A/\fp \otimes_k \Omega_Y \longrightarrow 0
\]
of sheaves on $Y_\et$. Taking global sections now yields an isomorphism
\[
	\H^1( Y_\fl, (A/\fp)^\rmD ) \isomto  A/\fp \otimes_k \Omega_{R/k}^{\cart=1},
\]
where $\Omega_{R/k}^{\cart=1}$ denotes the $k$-vector space of Cartier-invariant K\"ahler differentials.

On the other hand, we have a natural isomorphism
\[
	(A/\fp)^\rmD \isomto  A/\fp \otimes_{\bF_p} \mu_p,
\]
of sheaves on $Y_\fl$ and the Kummer sequence
\[
	1 \longrightarrow \mu_p \longrightarrow \bG_m \overset{p}{\longrightarrow}
	\bG_m \longrightarrow 1 
\]
gives rise to a short  exact sequence
\begin{equation}\label{seq1}
	0 \longrightarrow
	A/\fp \otimes_\bZ \Gamma( Y, \cO_Y^\times ) \longrightarrow
	\H^1( Y_\fl, (A/\fp)^\rmD )
	\longrightarrow
	A/\fp \otimes_{\bF_p}(\Pic Y)[p]
	\longrightarrow 0.
\end{equation}
The proof of Theorem \ref{thmduality} shows that the resulting composed morphism
\[
	A/\fp \otimes_\bZ \Gamma( Y, \cO_Y^\times ) \longrightarrow
	\H^1( Y_\fl, (A/\fp)^\rmD ) \isomto
	A/\fp \otimes_k \Omega_R^{\cart=1}
\]
is the map induced from
\[
	\dlog\colon \Gamma( Y, \cO_Y^\times ) \to \Omega_R^{\cart=1}\colon
	u \mapsto \frac{\diff u}{u},
\]
so that (\ref{seq1}) becomes the short exact sequence (\ref{kummer}).

\section{Comparing $(A/\fp)^\rmD$ and $C[\fp]^\rmD$-coefficients}\label{seccomparing}

Choose a nonzero torsion point $\lambda \in C[\fp](L)$.
Then $\lambda$ defines a morphism
$ (A/\fp)_Y \to C[\fp]_Y $ and hence a morphism of Cartier duals
\[
	C[\fp]^\rmD_Y \overset{\lambda}{\longrightarrow} (A/\fp)^\rmD_Y.
\]
Let $\fP \in Y$ be the unique prime above $\fp\subset A$. We have $\fP = R\lambda$.

\begin{proposition}\label{propfast}
The sequence
\begin{equation}{\label{eqfast}}
	0 \longto \R^1\!f_\ast C[\fp]^\rmD 
		\overset{\lambda}{\longto} \R^1\!f_\ast (A/\fp)^\rmD 
		\longto \Omega_Y/\fP^{q^d}\!\Omega_Y
		\overset{1-\cart^d}{\longto} \Omega_Y/\fP\Omega_Y
		\longto 0,
\end{equation}
of sheaves on $Y_\et$ is exact and if $i\neq 1$ then $\R^i\!f_\ast  C[\fp]^\rmD = 0$.
\end{proposition}

Note that for all $N$ the sheaf $\Omega_Y / \fP^N \Omega_Y$ on $Y_\et$ is naturally a sheaf of $A/\fp$-modules. The middle map in the proposition is the composition
\[
	\R^1\!f_\ast (A/\fp)^\rmD \longto A/\fp \otimes_k \Omega_Y \surjto
	\Omega_Y/\fP^{q^d}\!\Omega_Y.
\]
Taking global sections in (\ref{eqfast}) we obtain an exact sequence of $A/\fp$-vector spaces
\[
	0 \longrightarrow \H^1( Y_\fl, C[\fp]_Y^\rmD ) \overset{\lambda}{\longrightarrow}
	A/\fp\otimes_k \Omega_R^{\cart=1} \longrightarrow \Omega_R/\fP^{q^d}\!\Omega_R
\]
and considering the $G$-action on $\lambda$ we see that Proposition \ref{propfast} implies
 Theorem \ref{thmdiff}.

As one may expect, the proof of Proposition \ref{propfast} relies on a careful analysis of the group scheme $C[\fp]_Y$ near the prime $\fP$. 

Let $\bar{s}\to Y$ be a geometric point lying above $\fP\in Y$,  

\begin{lemma}
There is an \'etale neighborhood $V\to Y$ of $\bar{s}$ and a short exact sequence
\[
	0 \longto C[\fp]_V \longto \bG_{a,V}
	\overset{ \lambda^{q^d-1}-F^d }{ \longto } \bG_{a,V} \longto 0
\]
of sheaves of $A/\fp$-vector spaces on $V_\fl$.
\end{lemma}

\begin{proof}
Let $\cO_{Y,\bar{s}}$ be the \'etale stalk of $\cO_{Y}$ at $\bar{s}$ (a strict henselization of $\cO_{Y,\fP}$) and let $S=\Spec \cO_{Y,\bar{s}}$.
We have that $C[\fp]_S$ is a finite flat $A/\fp$-vector space scheme of rank $q^d$ over $S$, \'etale over the generic fibre. Such vector space schemes have been classified by Raynaud \cite[\S 1.5]{Raynaud74} (generalizing the results of Oort and Tate
\cite{Tate70}). Let $q=p^r$ with $p=\car k$, then the classification says that
$C[\fp]_S$ is a subgroupscheme of $\bG_a^{rd}$ given by equations
\[
	X_i^p = a_i X_{i+1}
\]
for some $a_i \in \cO_{Y,\bar{s}}$, and where the index $i$ runs over
$\bZ/rd\bZ$. Since the special fibre
of $C[\fp]_S$ is the kernel of $F^d$ on $\bG_a$, we find that all but one $a_i$ are units. In particular, we can eliminate all but one variable and find that 
$C[\fp]_S$ sits in a short exact sequence
\[
	0 \longto C[\fp]_S \longto \bG_{a,S} \overset{a - F^d}{\longto} \bG_{a,S} 
		\longto 0
\]
for some $a \in \cO_{Y,\bar{s}}$, well-defined up to a unit. We claim that
$a=\lambda^{q^d-1}$ (up to a unit). To see this, we compute the discriminant of the finite flat $S$-scheme $C[\fp]_S$ in two ways. On the one hand $C[\fp]_S$ is defined by the equation $X^{q^d}-aX$, with discriminant $a^{q^d}$ (modulo squares of units). On the other hand, $C[\fp]$ is the $\fp$-torsion scheme of the Carlitz module and hence it is given by an equation 
\[
	X^{q^d} + b_{d-1} X^{q^{d-1}} + \ldots + b_0 X
\]
with $b_i \in A$, and with $b_0$ a generator of $\fp$. In this way we find that the discriminant equals $b_0^{q^d}$ (modulo squares of units). Comparing the two expressions we conclude
that we can take $a=\lambda^{q^d-1}$, which proves the claim.

To finish the proof it suffices to observe that this short exact sequence is already defined over some \'etale neighbourhood $V\to Y$ of $\bar{s}$.
\end{proof}

Using this lemma we can now prove Proposition \ref{propfast}.

\begin{proof}[Proof of Proposition \ref{propfast}]
Let $V$ be as in the lemma and $U := Y - \fP$. Then $\{U,V\}$ is an \'etale cover of
$Y$ and it suffices to prove that the pull-backs of (\ref{eqfast}) to $U_\et$ and $V_\et$ are exact.

The pull-back to $U_\et$ is the sequence
\[
	0 \longto \R^1\!f_\ast C[\fp]_U^\rmD 
		\overset{\lambda}{\longto} \R^1\!f_\ast (A/\fp)_U^\rmD
	\longto 0
\]
which is exact because $\lambda\colon (A/\fp)_U \to C[\fp]_U$ is an isomorphism of sheaves on $U_\fl$.

For the exactness over $V_\et$, consider the commutative square
\[
\begin{CD}
\bG_{a,V} @>{1-F^d}>> \bG_{a,V}  \\
 @VV{\lambda}V @VV{\lambda^{q^d}}V @. \\
\bG_{a,V} @>{\lambda^{q^d-1}-F^d}>> \bG_{a,V} 
\end{CD}
\]
It extends to a map of short exact sequences
\[
\begin{CD}
0 @>>> (A/\fp)_V @>>> \bG_{a,V} @>{1-F^d}>> \bG_{a,V} @>>> 0 \\
@. @VVV @VV{\lambda}V @VV{\lambda^{q^d}}V @. \\
0 @>>> C[\fp]_V @>>> \bG_{a,V} @>{\lambda^{q^d-1}-F^d}>> \bG_{a,V} @>>> 0
\end{CD}
\]
and without loss of generality we may assume that the leftmost vertical map is the one
induced by $\lambda$. Now Theorem \ref{thmduality} (with $k$, $F$, and $S$ replaced by $A/\fp$, $F^d$ and $V$) yields a commutative diagram of sheaves of $A/\fp$-vector spaces on $V_\et$ with exact rows:
\[
\begin{CD}
0 @>>> \R^1\!f_\ast C[\fp]_V^\rmD @>>>
	\Omega_V @>{\lambda^{q^d-1}-\cart^d}>> \Omega_V @>>> 0 \\
@. @V{\lambda}VV @VV{\lambda^{q^d}}V @VV{\lambda}V @. \\
0 @>>> \R^1\!f_\ast (A/\fp)_V^\rmD
	@>>>  \Omega_V @>{1-\cart^d}>>\Omega_V @>>> 0
\end{CD}
\]
(where by abuse of notation, we denote the canonical maps of sites
 $V_\fl \to V_\et$ and $Y_\fl\to Y_\et$ by the same symbol $f$.)
This shows that on $V_\et$ we have an exact sequence
\[
	0 \longto \R^1\!f_\ast C[\fp]_V^\rmD 
		\overset{\lambda}{\longto} \R^1\!f_\ast (A/\fp)_V^\rmD 
		\longto \Omega_V/\lambda^{q^d}\Omega_V
		\overset{1-\cart^d}{\longto} \Omega_V/\lambda\Omega_V
		\longto 0,
\]
so the pullback of (\ref{eqfast}) to $V_\et$ is exact. 
\end{proof}

\section{A candidate cohomology class}\label{seccandidate}

Let $\lambda \in R$ be a primitive $\fp$-torsion point of the Carlitz module.
Consider the decomposition
\[
	1\otimes \lambda = \sum_{n=1}^{q^d-1} \lambda_n
\]
in $A/\fp \otimes_\bZ \Gamma(Y-\fP, \cO_Y^\times )$. 	
In this section we will prove Theorem \ref{thmcandidate}, which states that for $1\leq n < q^d-1$
we have 
\[
	\lambda_n \in A/\fp \otimes\Gamma( Y, \cO_Y^\times )
\]
and that the following are equivalent
\begin{enumerate}
\item $\fp$ divides $BC_n$;
\item $\dlog \lambda_n$ lies in the kernel of
	$A/\fp \otimes_k \Omega_R \surjto \Omega_R/\fP^{q^d}\Omega_R$.
\end{enumerate}

We start with the first assertion.

\begin{proposition}
If $1\leq n<q^d-1$ then $\lambda_n \in A/\fp \otimes_\bZ \Gamma( Y, \cO^\times )$.
\end{proposition}

\begin{proof}
For all integers $n$ we have 
\[
	\lambda_n = -\sum_{g\in G} \chi(g)^{-n} \otimes g\lambda.
\]
If moreover $n$ is not divisible by $q^d-1$ then $\sum_{g\in G} \chi(g)^{-n} = 0$ so that we can rewrite the above identity as
\[
	\lambda_n = -\sum_{g\in G} \chi(g)^{-n} \otimes \frac{g\lambda}{\lambda}.
\]
Since the point $\fP$ is fixed under $G$ it follows that for all $g\in G$ one has that $g\lambda/\lambda$ has valuation $0$ at $\fP$  and therefore for all $1\leq n < q^d-1$ we have
\[
	\lambda_n \in A/\fp \otimes\Gamma( Y, \cO_Y^\times ),
\]
as was claimed.
\end{proof}

Now let $L_\fP$ be the completion of $L$ at $\fP$ and $\fm$ the maximal ideal of its valuation ring $\cO_{Y,\fP}^\wedge$. Note that $\fm=(\lambda)$.

Consider the quotient $\fm/\fm^{q^d}$. It carries two $A$-module structures:
\begin{enumerate}
\item the \emph{linear} action coming from the $A$-algebra structure of $\cO_{Y,\fP}^\wedge$;
\item the \emph{Carlitz} action defined using $\varphi$.
\end{enumerate}
Also, the Galois group $G$ acts on $\fm/\fm^{q^d}$ and the action commutes with both $A$-module structures.

\begin{lemma}
Both actions on $\fm/\fm^{q^d}$ factor over $A/\fp$.
\end{lemma}

\begin{proof}
Note that $\fp \cO_{Y,\fP}^\wedge = \fm^{q^d-1}$. In particular the assertion is immediate for the linear action. For the Carlitz action, consider a generator $f$ of $\fp$. Then 
\[
	\varphi(f) = a_0 + a_1 F + \cdots + a_{d-1} F^{d-1} + F^d
\]
with $a_i \in \fp$ for all $i$. From this it follows that $\varphi(f)$ maps
$\fm\subset \cO_{Y,\fP}^\wedge$ into $\fm^{q^d}$, as desired.
\end{proof}

The Carlitz exponential series
\[
	e(z) = \sum_{n=1}^\infty e_n z^n \in K[[z]]
\]
has the property that for all $n<q^d$ the coefficient $e_n$ is $\fp$-integral, so
the truncated and reduced exponential power series
\[
	\bar{e}(z) = \sum_{n=1}^{q^d-1} e_n z^n \in (A/\fp)[[z]]/(z^{q^d})
\]
defines a $k$-linear map
\[
	\bar{e}\colon \fm/\fm^{q^d} \to \fm/\fm^{q^d}
\]
which is an isomorphism because it induces the identity map on the intermediate quotients 
$\fm^i/\fm^{i+1}$. Note that $\bar{e}$ is $G$-equivariant, as the coefficients $e_i$ of the Carlitz exponential lie in $K$.

\begin{lemma}\label{truncatedfunctional}
For all $x\in \fm/\fm^{q^d}$ and $a\in A$ we have $\bar{e}(ax)=\varphi(a)\bar{e}(x)$.
\end{lemma}

\begin{proof}
In $K[[z]]$ we have the identity
\[
	e(tz) = te(z) + e(z)^q
\]
of formal power series. Identifying coefficients on both sides we find that in
$(A/\fp)[[z]]/(z^{q^d})$
we have
\[
	\bar{e}(tz) = t\bar{e}(z) + \bar{e}(z)^q,
\] 
and we deduce that for all $a\in A$ and $x\in \fm/\fm^{q^d}$ we have
$\bar{e}(ax) = \phi(a)\bar{e}(x)$.
\end{proof}

Put $\bar\pi := \bar{e}^{-1}(\bar\lambda)$, where $\bar\lambda$ is the image of
$\lambda \in \fm$ in $\fm/\fm^{q^d}$. 

\begin{lemma}\label{lemmapi} For all $g\in G$ we have $g\bar\pi = \chi(g)\bar\pi$.
\end{lemma}

In other words $\bar\pi \in \fm/\fm^{q^d}(\chi)$.

\begin{proof}[Proof of Lemma \ref{lemmapi}]
Let $g\in G$ and $a\in A$ be so that $a$ reduces to $g$ in $G=(A/\fp)^\times$.
Since $\lambda$ is a $\fp$-torsion point of the Carlitz module we have that
\[
	g\bar\lambda = \varphi(a)\bar\lambda.
\]
Applying $\bar{e}^{-1}$ to both sides we find with Lemma \ref{truncatedfunctional} that
\[
	g\bar\pi = a\bar\pi
\]
and by definition $a\bar\pi$ equals $\chi(g)\bar\pi$.
\end{proof}

Choose a lift $\pi \in \fm$ of $\bar\pi$ such that $g\pi = \chi(g)\pi$
for all $g$. Then $\pi$ is a uniformizing element of $L_\fp$.

\begin{proposition}
Let $1 \leq n < q^d-1$. Then
\[
	\dlog \lambda_n =  (BC_n \pi^n + \delta) \dlog \pi
\]
for some $\delta \in \fm^{n+q^d-1}$.
\end{proposition}

\begin{proof}
Since $\bar\lambda = \bar{e}(\bar\pi)$ we have in $\cO_{Y,\fP}^\wedge$ the identity
\[
	\lambda = \sum_{n=1}^{q^d-1} e_n \pi^n + \delta_1
\]
for some $\delta_1 \in \fm^{q^d}$. Since $\diff \pi^n = 0$ for any $n$ divisible by $q$ we
find
\[
	\diff\lambda = (1 + \delta_2) \diff \pi
\]
for some $\delta_2 \in \fm^{q^d}$. Dividing both expressions we find
\[
	\dlog \lambda = \left(\sum_{n=0}^{q^d-2} BC_n \pi^n + \delta_3\right) \dlog \pi
\]
for some $\delta_3\in \fm^{q^d-1}$. Now the proposition follows from decomposing this identity in isotypical components, since $\dlog \pi$ is
$G$-invariant and $g\pi = \chi(g)\pi$ for all $g\in G$.
\end{proof}

We can now finish the proof of Theorem \ref{thmcandidate}.

\begin{proof}[Proof of Theorem \ref{thmcandidate}]
If $n>1$ then the Theorem follows from the above proposition. 
If $n=1$ we consider two cases. Either $q>2$ and then $BC_1=0$ and $\dlog \lambda_1=0$, or else
$q=2$ and then $\fp$ does not divide $BC_1$ and from the above $\pi$-adic expansion we
see that $\dlog \lambda_1$ does not map to zero in $\Omega_R/\fP^{q^d} \Omega_R$. In both
cases the theorem holds.
\end{proof}

\section{Vanishing of $\lambda_n$}\label{secvanishing}

Let $W$ be the ring of Witt vectors of $A/\fp$. For $a\in (A/\fp)^\times$ we denote by
$\tilde{a} \in W^\times$ the Teichm\"uller lift of $a$. Also, we denote by 
$\tilde\chi \colon G \to W^\times$ the Teichm\"uller lift of the character
$\chi \colon G \to (A/\fp)^\times$. 
If $M$ is a $W[G]$-module then it decomposes into isotypical components
\[
	M=\bigoplus_{n=1}^{q^d-1} M(\tilde\chi^n)
\]
with $G$ acting via $\tilde\chi^n$ on $M(\tilde\chi^n)$.

Put $U := W \otimes_\bZ \Gamma( Y, \cO_Y^\times )$ and let $D$ be the $W$-module of degree zero $W$-divisors on $X-Y$. Then we have a natural inclusion $U \shortinjto D$ with finite quotient.
Consider the decomposition of $1\otimes \lambda \in W \otimes \Gamma( Y-\fP, \cO_Y^\times)$ in
isotypical components:
\[
	1 \otimes \lambda = \sum_{n=1}^{q^d-1} \tilde{\lambda}_n \quad 
	\text{ with } \quad
	\tilde\lambda_n \in W \otimes_\bZ \Gamma( Y-\fP, \cO_Y^\times)(\tilde{\chi}^n).
\]
We have
\[
	\tilde\lambda_n = \sum_{g\in G} \chi(g)^{-n} \otimes g\lambda 
\]
and for $1<n<q^d-1$ we have that $\tilde\lambda_n$ lies in $U(\tilde\chi^n)$ and it maps to 
$\lambda_n$ under the reduction map
\[
	U \surjto 
	A/\fp \otimes_\bZ \Gamma( Y, \cO_Y^\times ).
\]
If $n$ is divisible by $q-1$ but not by $q^d-1$, the $W$-modules $D(\tilde\chi^n)$ and $U(\tilde\chi^n)$ 
are free of rank one. In particular
\[
	\lambda_n = 0 \, \text{ if and only if } \,
	\frac{ U(\tilde\chi^n) }{ W\tilde\lambda_n } \neq 0,
\]
and Theorem \ref{thmvanishing} follows from the following.

\begin{proposition}\label{proplengths}
Let $n$ be divisible by $q-1$ but not by $q^d-1$. Then the finite
$W$-modules
\[
	\frac{U(\tilde\chi^n)}{W\lambda_n}
\]
and
\[
	W \otimes_\bZ \Pic Y(\tilde\chi^n)
\]
have the same length.
\end{proposition}

\begin{proof}
Let $X$ be the canonical compactification of $Y$.
Since we have a short
exact sequence of $W$-modules
\[
	0 \longto
	\frac{  D(\tilde\chi^n)}{ U(\tilde\chi^n) }
	\longto
	W\otimes_\bZ (\Pic^0 X) (\tilde\chi^n) \to W \otimes_{\bZ}(\Pic Y)(\tilde\chi^n) 
	\longto 0,
\]
it suffices to show that
\[
	\frac{D(\tilde\chi^n)}{W\lambda_n}
	\, \text{ and } \,
	W\otimes_\bZ (\Pic^0 X)(\tilde\chi^n)
\]
have the same length. By Goss and Sinnott \cite{Goss85} the length of $W\otimes_\bZ (\Pic^0 X) (\tilde\chi^n)$ is the $p$-adic valuation of $L(1, \tilde\chi^{-n}) \in W$.
We will show that also the length of $D(\tilde\chi^n)/W\lambda_n$
equals the $p$-adic valuation of $L(1, \tilde\chi^{-n})$.

Since $n$ is divisible by $q-1$, the representation
$\tilde\chi^{-n}$ is unramified at $\infty$. Since all the points of
$X$ lying above $\infty$ are $k$-rational, the local $L$-factor at $\infty$ of
$L(T, \tilde\chi^{-n})$ is $(1-T)^{-1}$. Since $n$ is not divisible by $q^d-1$, the representation is ramified at $\fp$ and hence the local $L$-factor at $\fp$ is $1$.
Recall that for a prime $\fq\subset A$ coprime with $\fp$ we have that $\chi(\Frob_\fq)$ is the
image of the monic generator of $\fq$ in $(A/\fp)^\times$. Together with unique factorization in $A$ we obtain
\[
	L( T, \tilde\chi^{-n} ) = 
	(1-T)^{-1} 
	\sum_{a\in A_+, a\not\in \fp} \tilde{a}^{-n}T^{\deg a},
\]
where $A_+$ is the set of monic elements of $A$. In fact it is easy to see that for $m\geq d$
the coefficient of $T^m$ in the sum vanishes, so we have
\begin{equation}\label{Lformula}
	L( T, \chi^{-n} ) = 
	(1-T)^{-1}
	\sum_{a\in A^{<d}_+} \tilde{a}^{-n}T^{\deg a},
\end{equation}
where $A_+^{<d}$ is the set of monic elements of degree smaller then $d$.

Since $n$ is divisible by $q-1$ we have
\[
	\sum_{a\in A_+^{<d}} \tilde{a}^{-n}T^{\deg a}
		= \frac{1}{q-1} \sum_{a\in A^{<d}} \tilde{a}^{-n}T^{\deg a}.
\]
We conclude from (\ref{Lformula}) that
\[
	L(1,\tilde\chi^{-n}) = \frac{1}{q-1}\sum_{a\in A^{<d}} (\deg a)\tilde{a}^{-n}.
\]
Consider the function
\[
	\deg\colon G \to \left\{ 0,1,\ldots, d-1 \right\}
\]
which maps $g\in G$ to the degree of its unique representative in $A^{<d}$. Then the above identity can be rewritten as
\[
	L(1,\tilde\chi^{-n}) = \frac{1}{q-1} \sum_{g\in G} (\deg g)\tilde{g}^{-n}.
\]
By \cite[\emph{p.} 372]{Galovich81} there is a point in $X-Y$ with associated valuation $v$ 
and integers $u,w$ with $(u,p)=1$ such that  
\[
	v( g\lambda ) =  u \deg g + w
\]
for all $g\in G$. The valuation $v$ extends to an  isomorphism of $W$-modules
\[
	v\colon D(\tilde\chi^{n}) \to W,
\]
and we have
\begin{eqnarray*}
	v( \lambda_n ) &=& \sum_{g \in G} \tilde{g}^{-n} v(g\lambda) \\
			&=& u (q-1) L(1,\tilde\chi^{-n}) + w \sum_{g\in G} \tilde{g}^{-n} \\
			&=&  u (q-1) L(1,\tilde\chi^{-n}).
\end{eqnarray*}
In particular, the length of $D(\tilde\chi^n)/\lambda_n$ is the $p$-adic valuation of
$L(1,\tilde\chi^{-n})$ and the proposition follows.
\end{proof}

\section{Complement: the class module of $Y$}\label{complement}

Let $L$ be an arbitrary finite extension of $K$ and $R$ the integral closure of $A$ in $L$. 
Put $Y= \Spec R$.  In \cite{Taelman10b} and \cite{Taelman11} we have given several equivalent definitions of a finite $A$-module $H(C/Y)$ depending on $Y$, that is analogous to the class group of a number field.  One of these definitions is the following.

Let $X$ be the canonical compactification of $Y$ and let $\infty$ be the divisor on $X$ of zeroes of $1/t \in L$. (This is also the inverse image of the divisor $\infty $ on $\bP^1$.) Then $H(C/Y)$ is defined by the exact sequence
\begin{equation}\label{defsha}
	A \otimes_k \H^1( X, \cO_X ) \overset{\partial}{\longto}
	A \otimes_k \H^1( X, \cO_X(\infty) ) \longto H(C/Y) \longto 0,
\end{equation}
where
\[
	\partial = 1 \otimes (t+\rmF) - t\otimes 1.
\]

\begin{theorem} \label{thmcomparison}
Let $I\subset A$ be a nonzero ideal. Then there is a natural isomorphism
\[
	\H^1( Y_\fl, C[I]^\rmD )^\vee \isomto H(C/Y) \otimes_A A/I
\]
where $(-)^\vee$ denotes the $k$-linear dual.
\end{theorem}

\begin{proof}
The starting point of the proof is the exact sequence of sheaves of $A$-modules
\[
	0 \longrightarrow A \otimes_k \bG_{a} \overset{\partial}{\longrightarrow} 
		A \otimes_k \bG_{a} \overset{\alpha}{\longrightarrow}
	 	C \longrightarrow 0
\]
with $\partial( a\otimes f ) = a \otimes (f^q + tf) - ta \otimes f  $
and with $\alpha (a\otimes f) = \phi(a) f $. From this we derive
a short exact sequence 
\[
	0 \longrightarrow C[I]_Y \longrightarrow A/I \otimes_k \bG_{a} 
	\overset{\partial}{\longrightarrow} A/I \otimes_k \bG_{a} \longrightarrow 0.
\]
Using Theorem \ref{thmduality} we obtain a dual resolution: 
\[
	0 \longto \R^1\!f_\ast C[I]^\rmD \longto
		A/I \otimes_k \Omega_Y 
		\overset{\partial^\ast}{\longto}
		A/I \otimes_k \Omega_Y
	\longto 0
\]
of sheaves of $A$-modules on $Y_\et$, where $\partial^\ast = 1 \otimes (t+\cart) - t\otimes 1$.
Since 
$\R^i\!f_\ast C[I]^\rmD=0$ for $i\neq 1$,
taking global sections we obtain an exact sequence of $A$-modules
\begin{equation}\label{sha1}
	0 \longto \H^1( Y_\fl, C[I]^\rmD ) \longto
		A/I \otimes_k \Gamma( Y, \Omega_Y ) \overset{\partial^\ast}{\longto}
		A/I \otimes_k \Gamma( Y, \Omega_Y ).
\end{equation}

Now we claim that the natural inclusion of the complex
\[
	A/I \otimes_k \Gamma( X, \Omega_X(-\infty) ) \overset{\partial^\ast}{\longto}
	A/I \otimes_k \Gamma( X, \Omega_X )
\]
in the complex
\[
		A/I \otimes_k \Gamma( Y, \Omega_Y ) \overset{\partial^\ast}{\longto}
		A/I \otimes_k \Gamma( Y, \Omega_Y )
\]
is a quasi-isomorphism.
Indeed, the quotient has a filtration with intermediate quotients of the form
\[
	A/I \otimes_k \frac{  \Gamma( X, \Omega_X(n\infty) ) }
		{ \Gamma( X, \Omega_X((n-1)\infty) ) }
	\overset{\partial^\ast}{\longto}
	A/I \otimes_k \frac{\Gamma( X, \Omega_X((n+1)\infty) ) }
		{ \Gamma( X, \Omega_X(n\infty) ) }
\]
with $ n \in \bZ_{\geq 0}$. On these intermediate quotients we have that $1\otimes \cart$ and $t\otimes 1$ are zero, so that $\partial^\ast = 1\otimes t$, which is an isomorphism.

Hence we obtain from (\ref{sha1}) a new exact sequence
\[
	0 \longto \H^1( Y_\fl, C[I]^\rmD ) \longto
		A/I \otimes_k \Gamma( X, \Omega_X(-\infty) ) \overset{\partial^\ast}{\longto}
		A/I \otimes_k \Gamma( X, \Omega_X ).
\]
Under Serre duality the $q$-Cartier operator $\cart$ on $\Omega_X$ is adjoint to the $q$-Frobenius $\rmF$ on $\cO_X$, so we obtain a dual exact sequence
\[
	A/I \otimes_k \H^1( X, \cO_X ) \overset{\partial}{\longto} 
	A/I \otimes_k \H^1( X, \cO_X(\infty) ) \longto 
	\H^1( Y_\fl, C[I]^\rmD )^\vee \longto 0.
\]
Theorem \ref{thmcomparison} now follows by comparing this sequence with the sequence obtained by reducing (\ref{defsha}) modulo $I$.
\end{proof}

\bigskip
\bigskip
\small
\bibliographystyle{plain}
\bibliography{../master}

\end{document}